\documentclass[11pt]{amsart}
\usepackage{amsfonts}
\usepackage{amsmath}
\usepackage{amsthm}
\usepackage{amssymb}

\usepackage{multicol} %Pour diviser en colones

\usepackage{nccmath} % Pour \medint

\usepackage{eucal} %lettres calligraphiées

\usepackage[capitalise]{cleveref}

\usepackage{graphicx}

\usepackage{faktor}
\usepackage{xparse}
\usepackage{cleveref}

\let\ssection=\section

\newcommand{\p}{\noindent} 

\newtheorem{theo}{Theorem}[section]
\newtheorem*{theo*}{Theorem}
\newtheorem{lemm}[theo]{Lemma}
\newtheorem{defi}[theo]{Definition}
\newtheorem{exam}[theo]{Example}

\newtheorem{prop}[theo]{Proposition}
\newtheorem*{rema}{Remark}

\let\ssection=\section
\renewcommand{\section}{\setcounter{equation}{0}\ssection}

\newtheorem*{namedtheorem}{\theoremname}
\newcommand{\theoremname}{Theorem}
\newenvironment{named}[1]{\renewcommand{\theoremname}{#1}\begin{namedtheorem}}{\end{namedtheorem}}

\DeclareMathOperator{\PSL}{PSL}

\DeclareMathOperator{\vol}{vol}

\DeclareMathOperator{\supp}{Supp}
\DeclareMathOperator{\cotan}{cotan}

 \newcommand{\BH}{\mathbb H}
\newcommand{\BR}{\mathbb R} 
\newcommand{\BN}{\mathbb N} 
 \newcommand{\BZ}{\mathbb Z}

\newcommand{\BP}{\mathbb P}

 \newcommand{\CB}{\mathcal B}
\newcommand{\CC}{\mathcal C}

 \newcommand{\CL}{\mathcal L}
\newcommand{\CM}{\mathcal M}

 \newcommand{\CT}{\mathcal T}

\newcommand{\FC}{\mathfrak C}

\newcommand*{\Lfaktor}[2]{% \leftfaktor{#1}{#2} -> #2\#1
	\raisebox{-0.5\height}{\ensuremath{#2}}% Denominator
	\mkern-5mu\diagdown\mkern-4mu%BackSlash /
	\raisebox{0.5\height}{\ensuremath{#1}}% Numerator
}
\newcommand*{\Curve}[2]{\gamma^{({#1})}_{#2}}
\newcommand*{\hatCurve}[2]{\hat{\gamma}^{({#1})}_{#2}}

\begin{document}
	
	\title[Thurston's compactification via geodesic currents]{Thurston's compactification via geodesic currents: the case of non-compact finite area surfaces }
	\author{Marie  Trin}
	\address{Univ Rennes , CNRS, IRMAR - UMR 6625, F-35000 Rennes, France}
	\email{Marie.Trin@univ-rennes.fr}
	
	\thanks{The author is supported by Centre Henri Lebesgue, program ANR-11-LABX-0020-0 and R\'{e}gion Bretagne}
	\date{\today}
	
	\begin{abstract} In 1988, Bonahon gave a construction of Thurston's compactification of Teichm\"{u}ller space using geodesic currents. His argument only applies in the case of closed surfaces, and there are good reasons for that. We present a variant which applies to surfaces of finite area and to do so we prove a control theorem for sequences of random geodesics. Note that this theorem may be of independant interest, especially when the surface is non-compact.
	\end{abstract}
	
	\maketitle
	
	\section{Introduction}
	
	The Teichm\"{u}ller space $\CT(S)$ of a surface $S$ of finite topological type, with no boundary and of negative Euler characteristic $\chi(S)$ is the space of isotopy classes of (complete and finite volume) Riemannian metrics on $S$ of constant curvature $-1$. Teichm\"{u}ller space is not compact but Thurston showed in \cite{Thu} how it can be compactified by the space $\BP_+\CM\CL(S)$ of projective measured laminations on $S$. The starting point of Thurston's compactification is the embedding of $\CT(S)$ into the space $\BP_+(\BR_+^{\FC(S)})=\Lfaktor{\BR_+^{\FC(S)}}{\BR_+}$:
	\begin{center}
		$\begin{array}{ccccccccccc}
		\ell & : & \mathcal{T}(S) & \to     & \BP_+(\BR_+^{\FC(S)}) \\
		&   & X              & \mapsto & \BR_+\ell_X(\cdot);      \\
		\end{array}$
	\end{center}
	Here $\ell_X$ is the length function associated to the hyperbolic structure $X$ on $S$ and $\FC(S)$ is the set of free homotopy classes of essential closed curves of $S$. What Thurston did is to prove that the image of $\ell$ is locally compact and to identify the boundary of $\CT(S)$ in $\BP_+(\BR_+^{\FC(S)})$ with $\BP_+\CM\CL(S)$.
	
	\begin{theo*}[Thurston's compactification]
		If $S$ is a finite analytic type surface with negative Euler characteristic, then the accumulation points of $\CT(S)$ in $\BP_+(\BR_+^{\FC(S)})$ are the projective classes of functions $\gamma\mapsto i(\lambda,\gamma)$ where $\lambda\in\CM\CL(S)$ is a measured lamination on $S$.
	\end{theo*}
	Thurston's original proof is explained in \cite{ast}. Some versions using real-trees are given by Morgan–Shalen \cite{MS}, Bestvina \cite{Bestvina} or Paulin \cite{Paulin2}. An overview of the different compactification methods is availlable in \cite{Paulin} or \cite{Ohshika}. A compactification for the set of flat-structures and using geodesic currents is done in \cite{DLR}, note that this article is interested in both compact and non-compact surfaces.  Here, we will be mostly interested in a very elegant argument, for closed surfaces, due to Bonahon \cite{Bon88}. Let's sketch the proof. Recall that geodesic currents are $\pi_1(S)$-invariant Radon measures on the set of bi-infinite geodesics of the universal cover of $S$. Bonahon embeds $\CT(S)$ into the space $\CC(S)$ of geodesic currents of $S$, sending each element $X\in\CT(S)$ of the Teichm\"{u}ller space to the associated Liouville current $L_X\in\CC(S)$. The Liouville current satisfies two important properties:
	\begin{align}
	& i(L_X,\gamma)=\ell_X(\gamma) \quad \text{for every essential closed curve } \gamma\text{, and} \label{Prop1} \\
	& i(L_X,L_X)=\pi^2|\chi(S)|.\label{Prop2}
	\end{align}
	Here,  $i : \CC(S)\times\CC(S)\to\BR_ +$ is the intersection form, a continuous bilinear map extending the usual geometric intersection number between curves. Compactness of $S$ implies compactness of the space $\BP_+\CC(S)$ of projective currents. It follows that each sequence $(X_n)_{n\in\BN}$ in Teichm\"{u}ller space admits a subsequence, say the whole sequence, which projectively converges to a non-zero current $\mu$, meaning that there are positive reals $\varepsilon_n$ such that $\lim\limits_{n\to\infty}\varepsilon_nL_{X_n}=\mu$. The continuity of $i$ and property (\ref{Prop1}) ensure that the length functions $\ell_{X_n}(\cdot)$ converge projectively to $i(\mu,\cdot)$. Moreover, $\varepsilon_n$ tends to zero unless $X_n$ converges in $\CT(S)$. Knowing that $\varepsilon_n\xrightarrow[n\to\infty]{}0$, property (\ref{Prop2}) ensures that $i(\mu,\mu)=0$, meaning that $\mu$ is a measured lamination, as we needed to prove.

	We stress that Bonahon's argument, with all its simplicity, only applies to closed surfaces. We will come back later to this specificity and to the obstructions to a direct extension of his argument. Recently, Bonahon and \v{S}ari\'{c} have given another proof of this theorem using geodesic currents. The arguments in \cite{Bon21} are geared to infinite type surfaces, it is worth noticing that working in such a general context implies the lost of the simplicity of Bonahon's original proof.
	
	Our goal here is to adapt Bonahon's original argument to be able to deal with non-compact surfaces of finite analytic type .

	Let's look at the difficulties that prevent the extension of Bonahon's proof to the non-compact case. The intersection form, especially its continuity, is the linchpin of Bonahon's original proof. However, continuity fails when the surface is not compact, even if it has finite area (see \cite{DS} or Example \ref{NCI} below). We will therefore change our point of view to allow us to benefit from the continuity of $i$. We will consider currents on $\Sigma$ instead of $S$, where $\Sigma$ is a compact hyperbolic surface with geodesic boundary whose interior is homeomorphic to $S$, that is $S=\Sigma\setminus\partial\Sigma$. The second key ingredient of Bonahon's proof is the existence of the Liouville current but, as we will see, when working with currents on $\Sigma$ we lose the Liouville current.
	
	\begin{named}{Proposition~\ref{NLC}} Let $\Sigma$ be a compact hyperbolic surface with non-empty boundary and $X$ a hyperbolic structure on $S=\Sigma\setminus \partial\Sigma$. There is no current $L_X$ on $\Sigma$ which satisfies $i(L_X,\gamma)=\ell_X(\gamma)$ for every essential closed curve $\gamma\in\FC(\Sigma)$.
	\end{named}
	In order to recover a version of properties (\ref{Prop1}) and (\ref{Prop2}), we will, for every hyperbolic structure $X$ on $S$,  replace the Liouville current $L_X$ by specific sequences of random geodesics $(\gamma^{(X)}_n)_{n\in\BN}$, that is sequences of essential closed geodesics whose associated probability measures in $T^1X$ converge to the Liouville measure with respect to the weak-$*$ topology. They will be chosen to satisfy (\ref{Prop1}) and (\ref{Prop2}) asymptotically, that is:
	
	\begin{align}
	&\lim\limits_{n\to\infty} i\left(\frac{\gamma_n}{\ell_X(\gamma_n)},\gamma\right)=\frac{\ell_X(\gamma)}{\pi^2|\chi(S)|} \quad \text{for all essential closed curve }\gamma,\label{Prop1Bis}
	\end{align}
	
	\begin{align}
	&\lim\limits_{n\to\infty} i\left(\frac{\gamma_n}{\ell_X(\gamma_n)},\frac{\gamma_n}{\ell_X(\gamma_n)}\right)=\frac{1}{\pi^2|\chi(S)|}. \label{Prop2Bis}
	\end{align}
	As discussed in \cite{ES}, any sequence of random geodesics $(\gamma_n)_{n\in\BN}$ satisfies (\ref{Prop1Bis}). Moreover, if the surface is compact then (\ref{Prop2Bis}) is ensured for every sequence of random geodesics. However, for a non-compact surface,  arbitrary sequences of random geodesics do not necessarily satisfy (\ref{Prop2Bis}), see Example \ref{DivSelfInt} below. Indeed, a large part of this article will be dedicated to building sequences of random geodesics satisfying this property for non-compact surfaces. 
	
	\begin{theo}\label{Resultat}
		For every  complete and finite area hyperbolic structure $X$ on a finite analytic type surface $S$ of negative Euler characteristic $\chi(S)$, there is a sequence $(\Curve{X}{N})_{n\in\BN}$ of random geodesics such that:
		\begin{align*}
		\lim\limits_{n\to\infty} i\left(\frac{\Curve{X}{N}}{\ell_X(\Curve{X}{N})},\frac{\Curve{X}{N}}{\ell_X(\Curve{X}{N})}\right)=\frac{1}{\pi^2|\chi(S)|}.
		\end{align*}
	\end{theo}
	\noindent  Theorem~\ref{Resultat} is actually part of a more technical result, Theorem~\ref{Lemma}, that we will prove in section 3. The main additional content of Theorem~\ref{Lemma} is to ensure that the convergence rates in (\ref{Prop1Bis}) and (\ref{Prop2Bis}) hold with no dependance on the structure $X$. This uniformity will be important to achieve the proof of Thurston's compactification in section 4. Moreover, the proof of the theorem also ensures that we can control the behavior of sequences of random geodesics into some cusp's neighborhoods.

	\subsection*{Acknowledgements} 
	I am grateful to Juan Souto for our discussions and for all his suggestions about this paper. I would like to thank people who gave me their feedbacks on the first version of this document, especially Francis Bonahon, Frédéric Paulin, Beatrice Pozzetti, and Dylan Thurston for their remarks and Didac Martinez-Granado and Arya Vadnere for their questions.
	I also want to thank the PhD students of the IRMAR in ergodic theory for our debates on hyperbolic geometry, Barbara Schapira for her help with computations and  Jing Tao for the time she gave me, her advices and her comments on my work. Finally, I thank the anonymous referee for their careful review of this document.
	
	\section{Preliminaries}
	
	In this section, we give some technical results and definitions. We refer the reder to \cite{Bon88}, \cite{Bon86} and \cite{ES} for details. From now on,  let $S$ be a non-compact surface of finite analytic type, with negative Euler characteristic $\chi=\chi(S)<0$. We denote by $X,X',X_n...$ points in the Teichm\"{u}ller space of $S$, or maybe just the underlying complete and finite area hyperbolic structure. Note that, although not specified, all the hyperbolic structures are complete and finite area. Moreover, we will write $Z$ to refer indifferently to any finite area hyperbolic surface, possibly with punctures or with geodesic boundaries.
	If $S$ is endowed with a hyperbolic structure $X$ then every free homotopy class of essential closed curves contains a unique geodesic representative, so we identify a class with its geodesic representative when the hyperbolic structure is fixed. We will denote by $\FC(S)$ the set of free homotopy classes of essential closed curves -by essential we mean non-null-homotopic and non-peripheral- or equivalently the set of essential closed geodesics. Let also $\Sigma$ be a compact hyperbolic surface with geodesic boundary whose interior is homeomorphic to $S$. We fix a homeomorphism between $S$ and $\Sigma\setminus\partial\Sigma$. This homeomorphism immediately induces a correspondance between the essential closed curves of $S$ and the ones of $\Sigma$, that is 
	\begin{equation} \label{CorresCurves}
	\FC(S)=\FC(\Sigma).
	\end{equation}
	\p 
	The homeomorphism $S=\Sigma\setminus\partial\Sigma$ also gives an identification between measured laminations of $S$ and the ones of $\Sigma$ supported by $\Sigma\setminus\partial\Sigma$:
	\begin{equation}\label{CorresLam}
	\CM\CL(S)=\{\lambda\in\CM\CL(\Sigma)|\lambda \text{ supported by } \Sigma\setminus\partial\Sigma\}.
	\end{equation}
	The identifications (\ref{CorresCurves}) and (\ref{CorresLam}) will allow us to work on $\Sigma$ rather than on $S$.

	\subsection{Currents on surfaces}
	We recall now a few properties of currents that we will need in the following. A \textit{geodesic current} on $Z$ is a $\pi_1(Z)$-invariant Radon measure on the set of bi-infinite geodesics on the universal cover $\Tilde{Z}$ of $Z$ (even if the surface has non-empty boundary). The space $\CC(Z)$ of geodesic currents on $Z$ was introduced by Bonahon in \cite{Bon86} and is endowed with the weak-$*$ topology. For more information on currents we refer to \cite{Bon86}, \cite{Bon88}, \cite{AL} and, \cite[Chap. 3]{ES}.
	
	The currents we will be mainly interested in are weighted multicurves and measured laminations and we will always consider currents on the compact surface $\Sigma$. An advantage of doing so is that when $Z$ is compact, the topological space $\CC(Z)$ is locally compact, and the associated projective space $\BP_+\CC(Z)=\Lfaktor{\CC(Z)\setminus \{0\}}{\BR_+}$ is compact. Moreover, in the compact case, the geometric intersection number between curves extends to a continuous bilinear map $i : \CC(Z)\times\CC(Z)\to\BR_+$. It will be important later on to know that this form gives us a characterisation of the measured laminations as being the currents $\mu\in \CC(Z)$ satisfying $i(\mu,\mu)=0$. We can also notice that the boundary curves are characterised by a zero intersection form with every current. As mentioned earlier, the reason why we want to work with the currents on the compact surface $\Sigma$, rather than with the currents on $S$, is that the continuity of the intersection number fails in the latter case.
	
	\begin{exam}[Discontinuity of the intersection form in the non-compact case]\label{ObstructionPreuve}
		\label{NCI}
		\begin{figure}[!ht]
			\centering
			\includegraphics[scale=0.55]{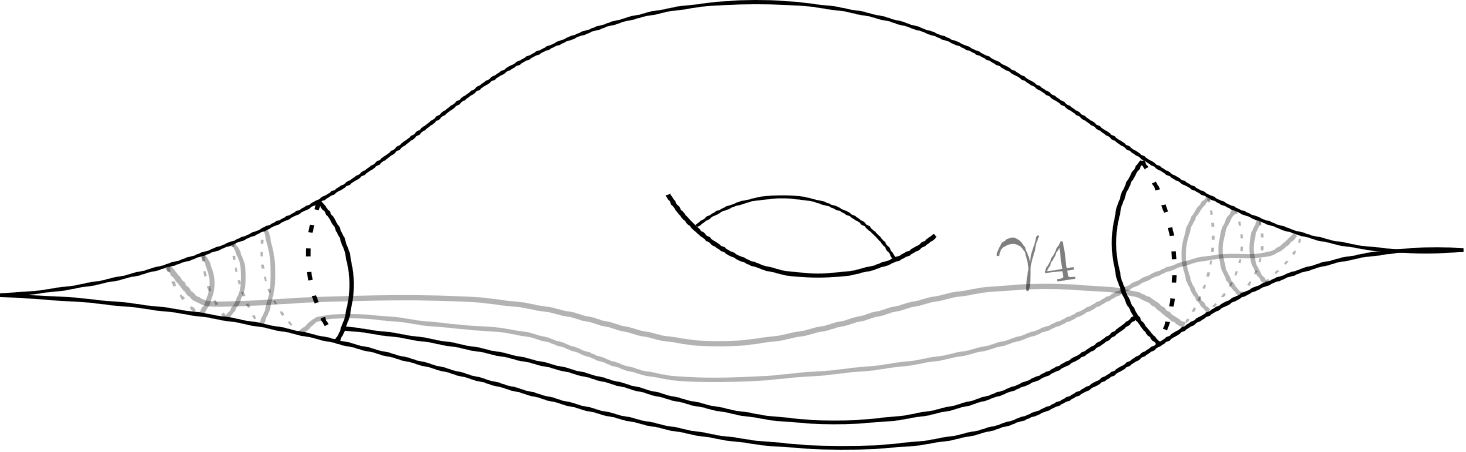}
			\caption{Obstruction to the continuity of $i$ \label{obstruction}}
		\end{figure}
		Take a hyperbolic surface with at least two cusps, fix an embedded horocycle around each of them, and a simple geodesic arc between those curves which meet them orthogonally. Note that this arc is part of a cusps-to-cusps geodesic arc $\gamma$. Consider a sequence of closed curves $(\gamma_n)_{n\in\BN}$, where $\gamma_n$ is the geodesic homotopic to the closed curve which runs the geodesic arc mentioned above, turns $n$ times around the first cusp following the fixed horocycle, goes back along the geodesic arc and turns $n$ times around the second cusp as in \cref{obstruction}. The self-intersection number of such a sequence is going to grow without bound. On the other hand, it approaches the weight 2 current associated to $\gamma$ which has $0$ self-intersection number.  
	\end{exam}	
	
	See \cite[Prop. 5.1]{DS} for a more detailed discussion on that obstruction to a continuous extension of the intersection number on the space of currents for non-compact surfaces.
	
	Example~\ref{ObstructionPreuve} shows that there is no continuous extension of the intersection number for currents on $S$ --- it is the reason why we chose to work with currents on the compact surface $\Sigma$ instead of the currents on $S=\Sigma\setminus\partial\Sigma$. This solves the problem of continuity of $i(\cdot,\cdot)$ but raises a new problem: we won't be able to consider the Liouville current anymore.
	
	\begin{prop}\label{NLC}
		Let $\Sigma$ be a compact hyperbolic surface with non-empty boundary and $X$ a hyperbolic structure on $S=\Sigma\setminus \partial\Sigma$. There is no current $L_X$ on $\Sigma$ which satisfies $i(L_X,\gamma)=\ell_X(\gamma)$ for every essential closed curve $\gamma\in\FC(\Sigma)$.
	\end{prop}
	
	\begin{proof}
		If $\gamma$ is a closed geodesic and $\mu$ a weighted multicurve of $\Sigma$ then 
		\begin{equation} \label{intersectionCourbes}
		i(\gamma,\mu)=\min
		\left\{ \sharp(\gamma'\cap\mu) ,
		\begin{split}
		\gamma'\text{ piecewise geodesic homotopic to  } \gamma \\ \text{ in }\mu\text{-general position}
		\end{split}
		\right\},
		\end{equation}
		where a piecewise geodesic homotopic to $\gamma$ is in $\mu$-general position if the set of geodesics passing through the corners has vanishing $\mu$ measure.
		\begin{figure}[!ht]
			\centering
			\includegraphics[scale=0.35]{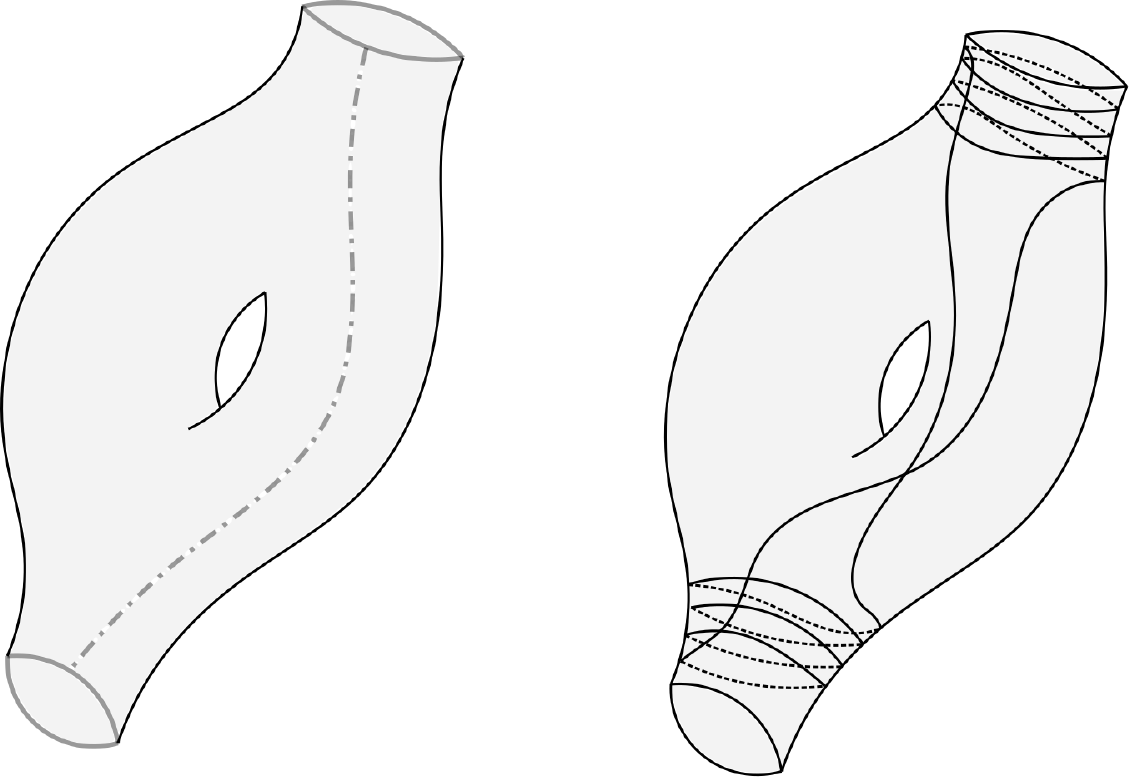}
			\caption{Obstruction to the existence of the Liouville current \label{obstructionLC}} 
		\end{figure}

		Now, consider $b_1$ and $b_2$ two boundary components of $\Sigma$, maybe the same, and $\gamma$ a non-trivial geodesic arc joining them. For every $k$, we define $\gamma_k$ as the unique closed geodesic homotopic to the piecewise geodesic which follows $\gamma$, turns $k$ times around $b_1$, follows back $\gamma$ and turns $k$ times around $b_2$. We obtain from \cref{intersectionCourbes} that for any weighted multicurve $\mu$, $$i(\gamma_k,\mu)\leq k\sharp(b_1\cap\mu)+k\sharp(b_2\cap\mu)+2\sharp(\gamma\cap\mu)=2\sharp(\gamma\cap\mu).$$

		We want to extend the previous inequality for $\mu$ a current, to do so we need a well-defined notion of intersection with $\gamma$. For this purpose we can embed $\Sigma$ into the closed doubled surface $D\Sigma$, for more details about how to pass from $\Sigma$ to $ D\Sigma$ the reader can refer to \cite{Tri}.  Hence, $\CC(\Sigma)$ is a subset of $\CC(D\Sigma)$, the double $\hat{\gamma}$ of $\gamma$ is a curve and in $\CC(D\Sigma)$ we have 
		\begin{align} \label{majoration}
		i(\gamma_k,\mu)\leq 2i(\hat{\gamma},\mu), 
		\end{align}
		for any $\mu$ weighted multicurve of $\Sigma$. Moreover, the weighted multicurves are dense in $\CC(\Sigma)$ and the intersection number is continuous in $\CC(D\Sigma)$ so \cref{majoration} induces that
		\begin{equation}\label{intersectionCourant}
		\forall \nu\in\CC(\Sigma),\quad i(\gamma_k,\nu)\leq2i(\hat{\gamma},\nu)<\infty.
		\end{equation}

		However,$\lim\limits_{k\to\infty} \ell_X(\gamma_k)=\infty$ for any hyperbolic structure $X$ on $S$, so \cref{intersectionCourant} prevents any intersection with a fixed current to produce the length.
	\end{proof}

	\subsection{Cusps neighborhoods and intersection number}
	
	Everything in the next section relies on a good understanding of the behaviour of geodesics in cusps.
	More precisely, if $X$ is a hyperbolic structure on $S$ then we denote by $H^i_{k}$ the embedded horosphere of length $1/k$ around the i-th cusp. The horosphere $H^i_{k}$ bounds the horoball $B^i_{k}$ of area $1/k$. We will refer to $H^i_{k}$ and $B^i_{k}$ as the horosphere and horoball of depth $k$. We also set $X^k$ the compact core of $X$ bounded by the horospheres $H^i_{k}$ and $\CB^k$ its complement:
	\begin{align}\label{defXB}
	X^k=X\setminus\bigcup\limits_{i}B^i_{k}, \quad \quad \CB^k=\bigcup\limits_{i}B^i_{k}.
	\end{align}
	There is a direct link  between the number of times a curve turns around a cusp and the depth it reaches \cite[Prop. 3.4]{BPT}. It follows that every curve that goes deep into a cusp has a large self-intersection number. To make this link more clear we recall a notion introduced in \cite[Def. 2.6]{ES}: the \textit{peripheral self-intersection number}.
	
	\begin{defi} \label{SelfIntNumb}
		Let $Z$ be a hyperbolic surface (compact or not) and recall that a peripheral subgroup of $\pi_1(Z)$ is nothing other than a cyclic subgroup generated by a non-essential closed curve. The \textbf{peripheral self-intersection number} $i_{per}(\gamma,\gamma)$ of $\gamma\in\FC(Z)$ is the supremum over all maximal peripheral subgroups $G \subset \pi_1(Z)$ of the maximal number of times that the image of a lift $\tilde{\gamma}$ of $\gamma$ under $\tilde{Z}\to\Lfaktor{\Tilde{Z}}{G}$ meets itself transversely.
	\end{defi}
	
	The peripheral self-intersection number is a topological invariant. It is thus independent of the metric on $S$, or more specifically, whether one considers the curves on $S$ or on $\Sigma$. Moreover, for every compact subset $K$ of $Z\setminus \partial Z$ there is a upper bound for the peripheral self-intersection number of the closed geodesics contained in $K$.  Conversely, for every $N>0$ there is a compact subset $K_N$ of $Z\setminus \partial Z$ that contains all the geodesics $\gamma$ with $i_{per}(\gamma,\gamma)\leq N$ \cite[Lem. 2.7]{ES}. In the absence of boundary, one can easily quantify this property.

	\begin{lemm}\label{ipbound}
		Let $X$ be a non-compact finite topological type surface with no boundary, and $\gamma$ be an essential closed curve on $X$, this curve has support on $X^k$ if and only if $i_{per}(\gamma,\gamma)\leq 4k$.
	\end{lemm} 
	
	\begin{proof}
		If we think of the curves of $\pi_1(X)$ as deck transformations then a peripheral subgroup of $\pi_1(X)$ is a subgroup generated by a parabolic element. Let's study a given cusp $C_i$, we can assume that the correspondence between $\tilde{X}$ and $\BH^2$ is such that an associated maximal parabolic element is $z\mapsto z+1$. In that case, $H_k^i$ lifts to the horizontal line $\{\Im(z)=k\}$ and if $\gamma$ is a closed geodesic of $X$ then the number of times that the image of a lift $\tilde{\gamma}$ under $\tilde{X}\to\Lfaktor{\Tilde{X}}{<z\mapsto z+1>}$ meets itself transversely is $\sharp\{n\in\BZ\setminus\{0\} | \tilde{\gamma}\cap(\tilde{\gamma}+n)\neq\emptyset\}$. However, $\gamma$ stays in $X^k$ around $C_i$, if and only if its lifts stay below the line $\{\Im(z)=k\}$, if and only if its lifts are half circles of radius at most $k$. Such a geodesic of $\BH^2$ meets at most $4k$ translations of itself ($n=\pm 1,\pm 2 ... \pm 2k$). The same process applies for every cusps and then to every maximal parabolic subgroup and we obtain the lemma.
	\end{proof}

	\section{Construction of controled sequences of random geodesics}
	
	In this section we prove that for all non-compact hyperbolic surfaces of finite volume with no boundary there are sequences of random geodesics satisfying (\ref{Prop2Bis}). However, we will first see with Example~\ref{DivSelfInt} that in the non-compact case not all the sequences of random geodesics have this property.

	\subsection{Sequences of random geodesics}
	As we saw in Proposition~\ref{NLC}, the Liouville current does not exist anymore in our setting. However, for every (complete and finite area) hyperbolic structure $X$ on $S$ the Liouville measure on $T^1X$ still exists. Recall that the \textit{Liouville measure} $\CL_X$ is the measure on the unit tangent bundle $T^1X$, obtained by pushing forward the Haar measure on $\PSL_2(\BR)$ and normalized so that $\CL_X(T^1X)=2\pi \vol_X(S)=4\pi^2|\chi(S)|$. We are going to consider geodesics approximating the Liouville measure in the following sense.
	
	\begin{defi}A sequence $(\gamma_n)_{n\in\BN}$ of essential closed geodesics on $X$ is a \textbf{sequence of random} \textbf{geodesics} if the associated probability measures converge to $\CL_X$ with respect to the weak-$*$ topology, meaning that:
		\begin{equation*}\medint\int_{T^1X}f\dfrac{d\gamma_n}{\ell_X(\gamma_n)} \underset{n\to+\infty}{\longrightarrow} \medint\int_{T^1X}f\dfrac{d\CL_X}{4\pi^2|\chi(S)|},
		\end{equation*}
		for every $f \in C_c^0(T^1X)$ continuous and compactly supported function on $T^1X$.  \end{defi} 
	
	\begin{rema}
		We will generally use the notation $\hat{\gamma}$ for the renormalisation $\dfrac{\gamma}{\ell_X(\gamma)}$.
	\end{rema}
	
	The Birkhoff ergodic theorem, together with the ergodicity of the geodesic flow, implies the existence of such sequences of geodesics. We refer to \cite[Chap. 2]{ES} for some facts about sequences of random geodesics that we will use here.
	
	The construction of the Liouville measure ensures that for a compact subsurface $K$ of $X$ we have $\CL_X(T^1K)=2\pi\vol_X(K)$. Then, if the boundary of $K$ is smooth, the Portmanteau Theorem implies that for every sequence of random geodesics $(\gamma_n)_{n\in\BN}$ we have
	\begin{equation*}
	\dfrac{\ell_X(\gamma_n\cap K)}{\ell_X(\gamma_n)}  \underset{n\to+\infty}{\longrightarrow} \dfrac{\vol_X(K)}{2\pi|\chi(S)|}.
	\end{equation*}
	Applying this property to our compact core $X^k$  we have
	\begin{equation} \label{areaComp}
	\dfrac{\ell_X(\gamma_n\cap X^k)}{\ell_X(\gamma_n)}  \underset{n\to+\infty}{\longrightarrow} \dfrac{\vol_X(X^k)}{2\pi|\chi(S)|},
	\end{equation}
	and hence, 
	\begin{equation} \label{areaHoro}
	\dfrac{\ell_X(\gamma_n\cap \CB^k)}{\ell_X(\gamma_n)}  \underset{n\to+\infty}{\longrightarrow} \dfrac{\vol_X(\CB^k)}{2\pi|\chi(S)|}.
	\end{equation}

	What is much more surprising is that sequences of random geodesics can also be used to compute lengths. More concretely, we have
	\begin{equation}
	\label{lengthSeg} \dfrac{i(\gamma_n,I)}{\ell_X(\gamma_n)} \underset{n\to+\infty}{\longrightarrow} \frac{\ell_X(I)}{\pi^2|\chi(S)|},\\
	\end{equation}
	for every compact geodesic segment $I$ in $X$. This property is basically due to Bonahon \cite[Prop. 14]{Bon88}, we also refer the reader to \cite[Prop. 2.4]{ES} for details. A direct consequence of (\ref{lengthSeg}) is that we can use random geodesics $(\gamma_n)_{n\in\BN}$ to compute the length of any essential geodesic $\gamma\in\FC(S)$:
	\begin{equation}\label{length}
	\dfrac{i(\gamma_n,\gamma)}{\ell_X(\gamma_n)} \underset{n\to+\infty}{\longrightarrow} \frac{\ell_X(\gamma)}{\pi^2|\chi(S)|}.
	\end{equation}
	Note that in this equation the curve $\gamma$ is fixed. Meaning that a priori, equation \eqref{length} does not say anything about $i(\gamma_n,\gamma_n)$. However, for compact sets (\ref{lengthSeg}) holds uniformly. As a consequence, cutting the geodesics $\gamma_n$ into geodesic segments we have 
	\begin{equation}
	\label{compact}
	i\left(\dfrac{\gamma_n}{\ell_X(\gamma_n)},\dfrac{\gamma_{n|K}}{\ell_X(\gamma_{n|K})}\right) \underset{n\to+\infty}{\longrightarrow} \dfrac{1}{\pi^2|\chi|}.
	\end{equation}
	for $K$ any fixed compact subsurface of $X$. 
	
	All those considerations about sequences of random geodesics apply to compact surfaces, hence, if $S$ were compact, applying (\ref{compact}) to $K=S$, then we would immediatly have that every sequence of random geodesics satisfies (\ref{Prop2Bis}). However, that is not necessarily true in general.
	\begin{exam}\label{DivSelfInt}
		First, note that an excursion of length $\ell$ into some $\CB_i^k$ has between $ke^{\ell/2}-2$ and $4ke^{\ell/2}$ self-intersections. Consider now a sequence of random geodesics $(\gamma_n)_{n\in\BN}$. Add to $\gamma_n$ an excursion of length $6\log(\ell_X(\gamma_n))$ at depth $k_n\xrightarrow[n\to\infty]{}\infty$ and pull it tight into a new geodesic $\gamma_n'$. If we add the excursions in a well-chosen way (for example, gluing it at the deepest point of an excursion) then the $(\gamma_n')_{n\in\BN}$ are still random geodesics  and 
		\[ \frac{i(\gamma'_n,\gamma'_n)}{\ell_X(\gamma'_n)^2} \approx \frac{i(\gamma_n,\gamma_n)+k_n\ell_X(\gamma_n)^3}{(\ell_X(\gamma_n)+6\log(\ell_X(\gamma_n)))^2}\underset{ +\infty}{\sim} \frac{i(\gamma_n,\gamma_n)}{\ell_X(\gamma_n)^2}+k_n\ell_X(\gamma_n)\xrightarrow[n\to\infty]{}\infty.\]
		One can can also refer to the arguments in Lemma~\ref{RandomGeod} below to prove that such sequences of random geodesics exist. 
	\end{exam}
	In \cite[Cor. 11.2]{Lalley2} or \cite{Lalley1}, Lalley gives a construction of random geodesics that justifies the use of the term "random": if for all $n$ the geodesic $\gamma_n$ is randomly chosen among the geodesics of length at most $n$ then $(\gamma_n)_{n\in\BN}$ is a sequence of random geodesics with probability 1. Hence, we wonder which proportion of sequences of random geodesics satisfies (\ref{Prop2Bis}). This problem might be linked to the study of the length of cusp excursions for random geodesics, see for example \cite{Haas}, \cite{Poll} or \cite{Sull} and the references therein.

	Anyway, the above example makes clear that to obtain (\ref{Prop2Bis}) in the non-compact case we have to control the excursions of the sequences of random geodesics into cusps neighborhoods. We will do it through the cutting process described below.

	\subsection{Cutting process}  Suppose that $X$ is a fixed complete and finite area hyperbolic structure for $S$. Recall that $X^t$ denotes the compact core of $X$ bounded by the horospheres of length $1/t$ around the cusps of $S$ and that $\CB^t=X\setminus X^t$ is its complement. Given two parameters $k\in\BN$ and $0<\theta<\pi/4$, and a curve $\gamma$ we want to cut the excursions of $\gamma$ in $\CB^k$ in order to prevent $\gamma$ from leaving $X^{k/\sin(\theta)}$. To do so, we will study $\gamma$ through its lifts in the universal cover $\tilde{X}$ of $X$. We focus here on a given cusp but we apply the same construction around each cusps of $X$. For $t\geq 1$ we denote by $H_t$ the horosphere of depth $t$ around this cusp and $B_t$ the horoball it bounds. Since $X$ is a hyperbolic surface endowed with a complete hyperbolic metric, its universal cover identifies with $\BH^2$, and we can suppose that the parabolic element associated to the cusp we are interested in is $z\mapsto z+1$. With this normalization $H_t$ lifts to the horizontal line $\{\Im(z)=t\}$ and we have that if a curve enters $H_t$ with some angle $\alpha\in\left[0,\pi/2\right)$ then it reaches the horosphere $H_{k/\sin(\alpha)}$ (we measure the non-oriented angle with the normal to the horosphere). We want to cut $\gamma$ in order to replace its long excursions into $B_k$ (\textit{ie.} the ones which cross $H_{k/\sin(\theta)}$) by short ones (excursions staying between $H_{k/\sin(2\theta)}$ and $H_{k/\sin(\theta)}$). To make it explicit we make a description of the process on the universal cover.

	\begin{figure}[!ht]
		\begin{minipage}[c]{.3\linewidth}
			\centering
			\includegraphics[scale=0.2]{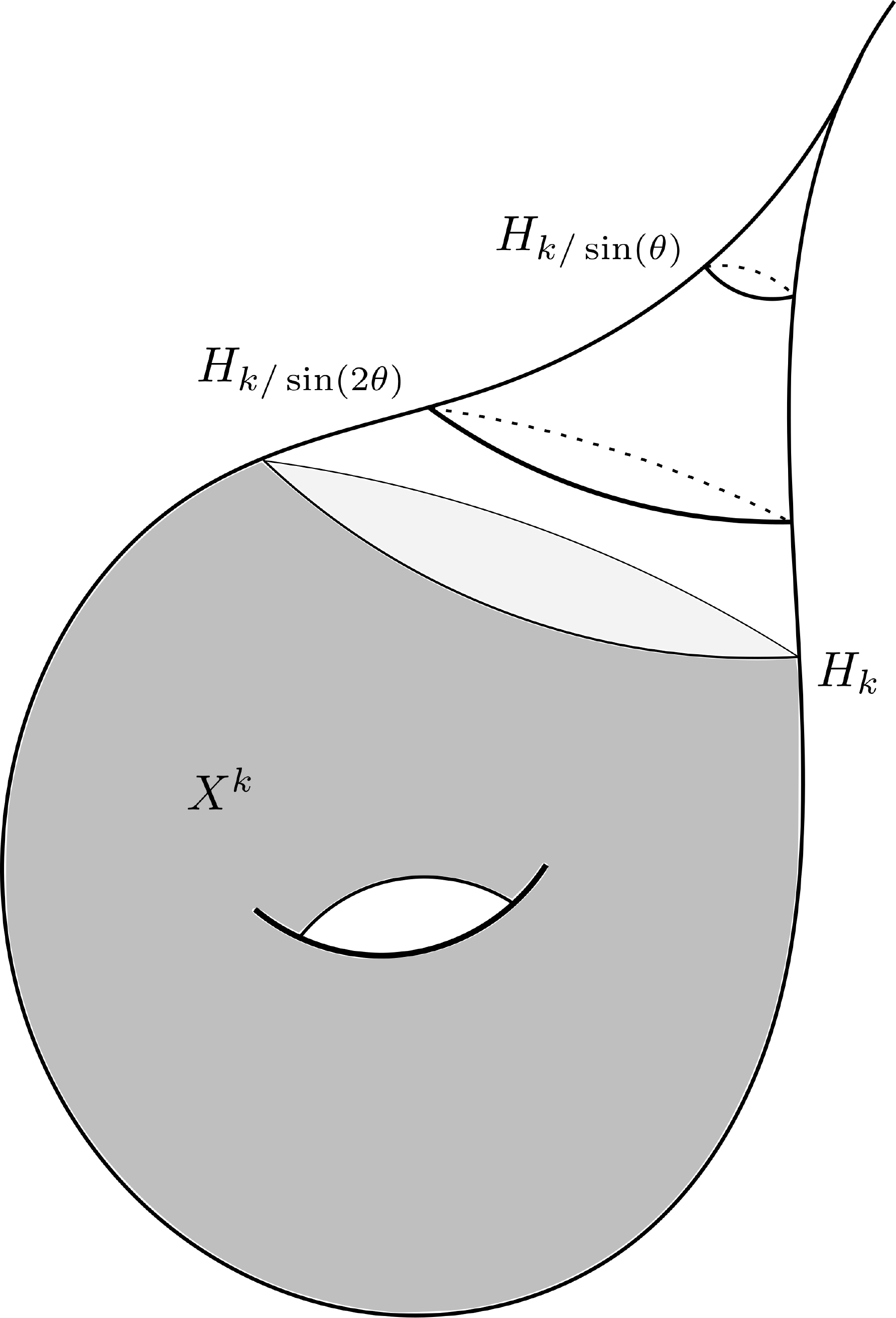}
		\end{minipage}
		\hfill%
		\begin{minipage}[c]{.6\linewidth}
			\centering
			\includegraphics[scale=0.32]{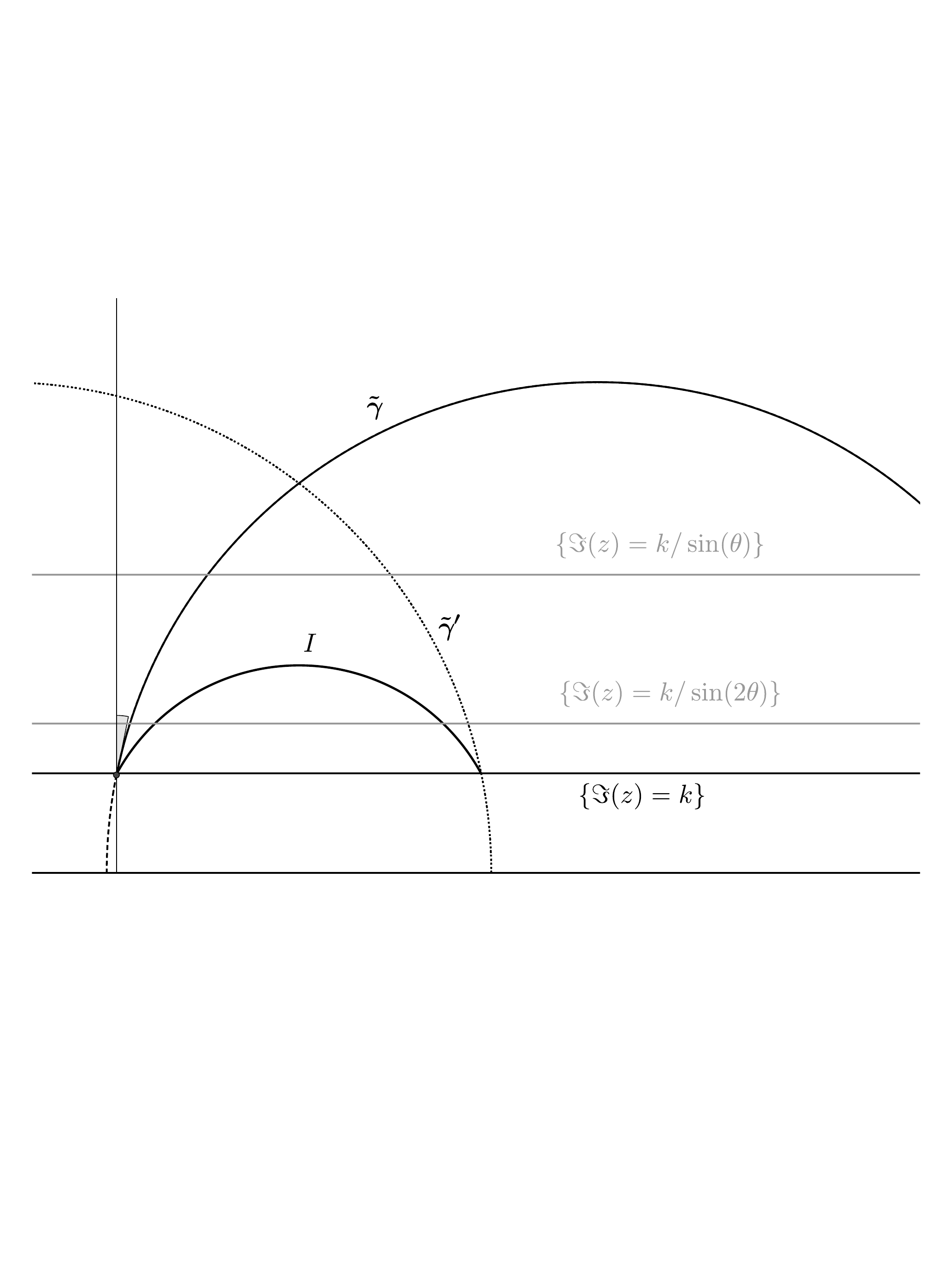}
		\end{minipage}
		\caption{Cutting process \label{Figure1}}
		
	\end{figure}
	
	If $\gamma$ makes excursions in $B_k$ we are going to modify $\gamma$ explaining the process on a fixed lift $\tilde{\gamma}$ which makes an excursion in the horoball $\{\Im(z)>k\}$ bounded by $\{\Im(z)=k\}$ but the same process applies to all lifts of $B_ k$. First, if $\tilde{\gamma}$ enters with an angle greater than $\theta$ then we don't change it. On the other hand, if it enters with an angle smaller than $\theta$ then we replace this arc by a geodesic arc $I$ which enters with angle between $\theta$ and $2\theta$ and whose exit point coincides with the exit point of a different lift $\tilde{\gamma}'$ of $\gamma$ (see \cref{Figure1}). This is always possible as long as $2k\cotan(\theta)-2k\cotan(2\theta)\geq 1$.  If we apply the same process to all the excursions of $\gamma$ around every cusp then $\gamma$ is replaced by a closed piecewise geodesic $\gamma'$. 
	
	Now, pulling $\gamma'$ tight we obtain a closed geodesic $\gamma^*$: we refer to $\gamma^*$ as the \textit{geodesic obtained by cutting process of parameters $k$ and $\theta$ from $\gamma$}. Note that if $\theta$ is small then $\gamma'$ and $\gamma^*$ have basically the same length, more precisely,  they can be mapped one to each other through a homotopy with small displacement and without disturbing to much the lengths. For the lengths, it is easy to see that there is some $e_\theta\xrightarrow[\theta\to 0]{}0$, idependent from $X$, such that for every $k\geq 1$ and $\theta$ small 
	\begin{equation} \label{ratio}
	\ell_X(\gamma')\leq(1+e_\theta)\ell_X(\gamma^*).
	\end{equation}
	Here $\ell_X(\gamma')$ refer to the arc length of $\gamma'$,  we will use again this abuse of notation but its meaning is clear from the context.

	\subsection{Construction of controled sequences of random geodesics}
	
	\begin{lemm} \label{CompLength}
		There is some $\theta_0>0$ such that if $(\gamma_n)_{n\in\BN}$ is a sequence of random geodesics on $X$ and $(\gamma_{n}^*)_{n\in\BN}$ is obtained from the $\gamma_n$ applying the cutting process of parameters $k>1$ and $\theta>\theta_0$ then there is $\mu_n\xrightarrow[n\to\infty]{}0$ such that 
		\begin{equation*}
		1\leq\dfrac{\ell_X(\gamma_n)}{\ell_X(\gamma_{n}^*)}\leq (1+\mu_n)\dfrac{\vol_X(S)}{\vol_X(X^{k})}(1+e_\theta),
		\end{equation*}
		for every $n$. Here, $e_\theta$ is as in \eqref{ratio}.
	\end{lemm}
	
	\begin{proof}
		We use the same notation as in the description of the cutting process, and, as above, we denote by $\ell_X(\gamma_n')$ the arc length of the piecewise geodesics.
		
		We take $\theta_0$ small enougth such that \eqref{ratio} occurs.
		The $\gamma_n$ being random geodesics, (\ref{areaComp}) ensures that we can find a sequence $\mu_n\xrightarrow[n\to\infty]{}0$ such that 	$\frac{\ell_X(\gamma_n)}{\ell_X(\gamma_{n|X^{k}})}=(1+\mu_n)\frac{\vol_X(S)}{\vol_X(X^{k})}$. The construction of $\gamma_n'$ ensures that $\gamma_{n|X^{k}}=\gamma'_{n|X^{k}}$, thus $\frac{\ell_X(\gamma_{n|X^{k}})}{\ell_X(\gamma_n')}\leq 1$ and if $\theta>\theta_0$ then $\frac{\ell_X(\gamma_n')}{\ell_X(\gamma^*_{n})}\leq(1+e_\theta)$. The upper bound follows from those three inequalities.
		
		Now, $\gamma_n$ and $\gamma'_n$ coincide on $X^k$ but $\gamma'_n$ has shorter excursions than $\gamma_n$ in $\CB^k$, hence, $\frac{\ell_X(\gamma_n)}{\ell_X(\gamma'_n)}\geq 1$. The geodesic $\gamma_{n}^*$ is the unique geodesic representative of the free homotopy class of $\gamma'_n$ which proves that $\frac{\ell_X(\gamma'_n)}{\ell_X(\gamma_{n}^*)}\geq 1$ and the lower bound follows.
	\end{proof}

	\begin{lemm}\label{RandomGeod}
		Let $(\gamma_n)_{n\in\BN}$ be a sequence of random geodesics. If $(\gamma_{n}^*)_{n\in\BN}$ is obtained from $(\gamma_n)_{n\in\BN}$ applying the cutting processes of parameters $k_n\xrightarrow[n\to\infty]{}\infty$ and $\theta_n\xrightarrow[n\to\infty]{} 0$, then $(\gamma_{n}^*)_{n\in\BN}$ is a sequence of random geodesics.
	\end{lemm}
	\begin{proof}
		In this proof, we denote by $\tilde{\gamma}$ the canonical lift of a geodesic $\gamma$ to the unit tangent bundle of $X$.
		
		Let $f\in C^0_c(T^1X)$ be a continuous and compactly supported function on $T^1X$, there is $K$ a compact core of $X$ such that $\supp(f)\subset T^1K$. Since $k_n\xrightarrow[n\to\infty]{} \infty$ then there is $n_0\in\BN$ such that for all $n\geq n_0, \quad\gamma_{n|K}=\gamma'_{n|K}$.   
		The homotopy between $\gamma'_n$ and $\gamma^*_n$ induces that the arcs of $\gamma_{n|K}$ are freely homotopic to geodesic arcs of $\gamma^*_n$. Such a homotopy induces a projection from $\gamma_{n|K}$ to $\gamma^*_{n}$ and lifts to $\Psi_n : \tilde{\gamma}_{n|K} \to \tilde{\gamma}^*_{n}$, which is a homeomorphism on its image. The homotopy can be chosen to have low displacement, that is $d(p,\Psi_n(p))\leq\varepsilon_n\xrightarrow[n\to\infty]{}0$ for every $p\in\tilde{\gamma}_{n|K}$, and not to distort too much the lengths. Moreover, we can find $\varphi_n : [0,\ell_X(\gamma_{n|K})] \to \BR_+$ a piecewise smooth reparametrization of $[0,\ell_X(\gamma_{n|K})]$ such that for all $t\in[0,\ell_X(\gamma_{n|K})]$, $\Psi_n(\tilde{\gamma}_{n|K}(t))=\tilde{\gamma}^*_{n}(\varphi_n(t))$. The homotopy between $\gamma'_n$ and $\gamma^*_n$ does not distort too much the lengths, hence, we have some $\delta_n \xrightarrow[n\to\infty]{}0$ such that $1-\delta_n\leq \varphi_n'\leq 1+\delta_n$ where it is defined.

		Fix some $\mu>0$. A compactly supported continuous function is uniformly continuous, thus, there is $\varepsilon_\mu>0$ such that if $d(p,q)\leq \varepsilon_\mu$ then $|f(p)-f(q)|\leq \mu$. We can suppose that for every $n\geq n_0$, $\varepsilon_n\leq \varepsilon_\mu$. We have 
		\begin{equation*}
		\int\limits_{T^1X}fd\gamma^*_{n}= \int_0^{\ell_X(\Psi_n(\gamma_{n|K}))}f\circ\tilde{\gamma}^*_{n}(t)dt=\int_0^{\ell_X(\gamma_{n|K})}f\circ\tilde{\gamma}^*_{n}(\varphi_n(s))\varphi_n'(s)ds,
		\end{equation*}
		it follows that
		{\tiny
			\begin{align*}
			&\quad(1-\delta_n)\int_0^{\ell_X(\gamma_{n|K})}f(\Psi_n(\tilde{\gamma}_{n|K}(s))ds\leq\int\limits_{T^1X}fd\gamma^*_{n}\leq (1+\delta_n)\int_0^{\ell_X(\gamma_{n|K})}f(\Psi_n(\tilde{\gamma}_{n|K}(s))ds\\
			&\Rightarrow (1-\delta_n)\left( \int\limits_{T^1X}f d\gamma_n -\mu\ell_X(\gamma_{n|K})\right)\leq \int\limits_{T^1X}fd\gamma^*_{n} \leq(1+\delta_n)\left(\int\limits_{T^1X}f d\gamma_n +\mu\ell_X(\gamma_{n|K})\right)\\
			&\Rightarrow(1-\delta_n)\frac{\ell_X(\gamma_n)}{\ell_X(\gamma^*_{n})} \left(\int\limits_{T^1X}f d\hat{\gamma}_n  -\mu \right)\leq \int\limits_{T^1X}fd\hat{\gamma}^*_{n} \leq (1+\delta_n)\frac{\ell_X(\gamma_n)}{\ell_X(\gamma^*_{n})}\left( \int\limits_{T^1X}f d\hat{\gamma}_n +\mu\right)
			\end{align*}}
		Adapting the proof of Lemma~\ref{CompLength}  we have  $\dfrac{\ell_X(\gamma_n)}{\ell_X(\gamma^*_{n})}\xrightarrow[n\to\infty]{}1$, and passing to the limit in $n$ we obtain 
		{\small
			\begin{equation*}
			\int_{T^1X}f\dfrac{d\CL_X}{4\pi^2|\chi(S)|}-\mu\leq\underline{\lim}_n\int\limits_{T^1X}fd\hat{\gamma}^*_{n}\leq \overline{\lim}_n \int\limits_{T^1X}fd\hat{\gamma}^*_{n}\leq \int_{T^1X}f\dfrac{d\CL_X}{4\pi^2|\chi(S)|}+\mu.
			\end{equation*}}
		This is true for all $\mu$, hence, $\lim\limits_{n\to\infty} \medint\int\limits_{T^1X}fd\hat{\gamma}^*_{n} =\medint \int_{T^1X}f\dfrac{d\CL_X}{4\pi^2|\chi(S)|}$ and we have proved that $(\gamma^*_{n})_{n\in\BN}$ is a sequence of random geodesics.
	\end{proof}

	Now, for every hyperbolic structure $X$ on $S$, we will be able to build sequences $(\Curve{X}{n})_{n\in\BN}$ of random geodesics satisfying (\ref{Prop2Bis}). Moreover, we will build them in such a way that neither the converging rates in (\ref{Prop2Bis}) and (\ref{length}), nor the peripheral self-intersection numbers $i_{per}(\Curve{X}{n},\Curve{X}{n})$ depend on $X$.
	\begin{theo}\label{Lemma}
		For every complete and finite area hyperbolic structure $X$ on a finite analytic type surface of negative Euler characteristic $S$, there is a sequence $(\Curve{X}{n})_{n\in\BN}$ of random geodesics such that :
		\begin{align*}
		\lim\limits_{n\to\infty} i\left(\frac{\Curve{X}{n}}{\ell_X(\Curve{X}{n})},\frac{\Curve{X}{n}}{\ell_X(\Curve{X}{n})}\right)=\frac{1}{\pi^2|\chi(S)|}.
		\end{align*}
		More precisely, they can be chosen such that
		\begin{enumerate}
			\item $i(\hatCurve{X}{n},\hatCurve{X}{n})\leq \dfrac{1}{\pi^2\left|\chi(S)\right|}\left(1+\dfrac{1}{n}\right),\,\forall n \in\BN,$
			\item $\forall \alpha \in\FC(S),\,\exists n_\alpha\in\BN:\,\left|i(\hatCurve{X}{n},\alpha)\left(\frac{\ell_X(\alpha)}{\pi^2|\chi|}\right)^{-1}-1\right|\leq \dfrac{3}{n},\, \forall n\geq n_\alpha,$
			\item $ i_{per}(\Curve{X}{n},\Curve{X}{n})\leq C_n,\, \forall n \in\BN,$
		\end{enumerate}
		where $C_n$ and $n_\alpha$ do not depend on $X$. 
	\end{theo}
	
	\begin{proof}
		To obtain the desired sequence $(\Curve{X}{n})_{n\in\BN}$ we start with an arbitrary sequence of random geodesics $(\gamma_n)_{n\in\BN}$. For every $p$ we set $k_p=e^{p/2}$ and $\theta_p=e^{-p/2}$, if we apply the cutting process with parameters $k_p$ and $\theta_p$ to the sequence $(\gamma_n)_{n\in \BN}$ then we obtain a sequence $(\Tilde{\gamma}_n^p)_{n\in \BN}$ of piecewise geodesics and by pulling it tight a sequence $(\gamma_n^p)_{n\in\BN}$ of geodesics. We will chose the $(\Curve{X}{N})_{N\in\BN}$ among the $\gamma_n^p$.

		First, study the self-intersection number of those $\gamma_n^p$. As $\gamma_n^p$ is the geodesic representative of $\Tilde{\gamma}_n^p$, its self-intersection number is lower than the number of self-intersections of $\Tilde{\gamma}_n^p$. To count it, we divide $X$ into two parts, the compact core $X^{k}$ and its complement $\CB^{k}$. On  $X^{k}$, the geodesic arcs $\Tilde{\gamma}_{n|X^k}^p$ and $\gamma_{n|X^k}$ are identical so $\Tilde{\gamma}_n^p$ has $i(\gamma_{n|X^k},\gamma_n)$ self-intersections. On the complement, we count the self-intersections of $\tilde{\gamma}_n^p$ considering its different excursions in $\CB^k$: 
		\begin{align*} i(\gamma_n^p,\gamma_n^p) & \leq i(\gamma_n\cap X^k,\gamma_n) + \sum\limits_{I,J \text{ excursions in } \CB^k} i(I,J).                           
		\end{align*}
		We can distinguish two types of pairs $(I,J)$: the ones where at least one of the excursions stays in $\CB^k\cap X^{k/\sin(2\theta)}$, and the ones where both $I$ and $J$ reach $\CB^{k/\sin(2\theta)}$. In the first case, $I$ and $J$ meet at most as many times as the corresponding excursions of $\gamma_n$ and then:
		\begin{align*} i(\gamma_n^p,\gamma_n^p) & \leq i(\gamma_n\cap X^{k/\sin(2\theta)},\gamma_n) + \sum\limits_{\substack{I,J \text{ excursions in } \CB^k\\ \text{which reach } \CB^{k/\sin(2\theta)}}} i(I,J).                           
		\end{align*}
		\p 
		Moreover, an excursion of $\tilde{\gamma}_n^p$ in $\CB^k$ which reaches $\CB^{k/\sin(2\theta)}$ has a length of at least $\ln(1/\theta)$, a lower bound for the length of the geodesic arc which enters with angle $2\theta$. It follows that there is at most $\frac{\ell_X(\gamma_n\cap \CB^k)}{\ln(1/\theta)}$ such excursions. Also, the intersection number of two excursions reaching $\CB^{k/\sin(2\theta)}$ is at most  $4k/\theta$, the self-intersection number of the excursion which enters with angle $\theta$. All in all, 
		\begin{align*} i(\gamma_n^p,\gamma_n^p) &  \leq i(\gamma_n\cap X^{k/2\theta},\gamma_n) + \left( \frac{\ell_X(\gamma_n\cap \CB^k)}{\ln(1/\theta)} \right)^2 \frac{4k}{\theta}.
		\end{align*}
		\p 
		Applying equations (\ref{compact}) and (\ref{areaHoro}) we have 
		\begin{align*}
		& i(\gamma_n\cap X^{k/2\theta} ,\gamma_n)=(1+\varepsilon_n^p)\frac{\ell_X(\gamma_n)\ell_X(\gamma_n\cap X^{k/2\theta})}{\pi^2|\chi|}     \text{ and}  \\
		& \ell_X(\gamma_n\cap \CB^k)=(1+\delta_n^p)\frac{\ell_X(\gamma_n)}{2\pi|\chi|}\frac{C}{k} \quad \text{where $C$ is the number of cusps of $S$,} 
		\end{align*}
		where $\varepsilon_n^p\xrightarrow[n\to\infty]{}0$ and $\delta_n^p\xrightarrow[n\to\infty]{}0$ depend on $X$.
		As a consequence, 
		$$ i(\gamma_n^p,\gamma_n^p)\leq (1+\varepsilon_n^p)\frac{\ell_X(\gamma_n)^2}{\pi^2|\chi|}+\left((1+\delta_n^p)\frac{C\ell_X(\gamma_n)}{2\pi|\chi|\cdot k \cdot \ln(1/\theta)}\right)^2 \frac{4k}{\theta} $$
		and we obtain a upper bound for the self-intersection number of the normalized curves:
		\begin{align}
		i\left(\frac{\gamma_{n}^p}{\ell_X(\gamma_{n}^p)},\frac{\gamma_n^p}{\ell_X(\gamma_{n}^p)}\right)\leq \frac{1}{\pi^2|\chi|} \left((1+\varepsilon_n^p)+(1+\delta_n^p)^2\frac{C^2}{|\chi|}\frac{4}{p^2}  \right)\left(\frac{\ell_X(\gamma_n)}{\ell_X(\gamma_n^p)}\right)^2.\label{autoi}
		\end{align}
		
		We next study the intersection number of the $\gamma_n^p$ with closed curves. The set $\FC(S)$ is infinite and can be enumerated with $\FC(S)=\{\alpha_q|q\in\BN\}$ in such a way that $i_{per}(\alpha_q,\alpha_q)\leq 4q$ for every $q$. This enumeration is fixed whatever the structure $X$. Recall that for every $p$ we have $k=k_p=e^{p/2}$ and $\theta=\theta_p=e^{-p/2}$. Hence, when $p$ is big enough, for $q\leq p$ the curve $\alpha_q$ is included in $X^q\subset X^{k}$. However in $X^{k}$ we have $i(\gamma_n,\cdot)=i(\gamma_n^p,\cdot)$ thus 
		\begin{equation}
		i(\hat{\gamma}_n^p,\alpha_q)=\dfrac{\ell_X(\gamma_n)}{\ell_X(\gamma_n^p)}i(\hat{\gamma}_n,\alpha_q).\label{EgualIntersection}
		\end{equation}

		Now, applying Lemma~\ref{CompLength}, for every $p$ there is $\mu_n^p\xrightarrow[n\to\infty]{}0$, depending on $X$, such that
		
		\begin{equation}1\leq \frac{\ell_X(\gamma_{n})}{\ell_X(\gamma^p_{n})}\leq
		(1+\mu_n^p)
		\frac{\vol_X(S)}{\vol_X(X^{k})}(1+e_p). \label{RappLong} \end{equation}
		with $e_p=e_{e^{-p/2}}$ with the notation of \eqref{ratio}.

		Therefore, there are $m_p$ large enough such that $\varepsilon_{m_p}^p,\, \delta_{m_p}^p,\, \mu_{m_p}^p\leq\frac{1}{p}$, and $\left|\dfrac{i(\hat{\gamma}_{m_p},\alpha_q)}{\ell_X(\alpha_q)/\pi^2|\chi| }-1\right|\leq \frac{1}{p}$ for every $q\leq p$. Thus (\ref{RappLong}) and (\ref{autoi}) give us
		
		\begin{align}
		& 1\leq \frac{\ell_X(\gamma_{m_p})}{\ell_X(\gamma^p_{m_p})}\leq
		\frac{\vol_X(S)}{\vol_X(X^{k})}
		(1+\frac{1}{p})(1+e_p)\xrightarrow[p\to\infty]{}\,  1, \label{conv1}
		\end{align}
		\begin{multline}
		i(\hat{\gamma}^p_{m_p},\hat{\gamma}^p_{m_p})\leq \frac{1}{\pi^2|\chi|}\left(1+(1+\frac{1}{p})\frac{4C^2}{p^2|\chi|} \right)
		\frac{\vol_X(S)}{\vol_X(X^{k})}
		(1+\frac{1}{p})^2(1+e_p) \label{conv2}\\
		\xrightarrow[p \to \infty ]{} \frac{1}{\pi^2|\chi|}.
		\end{multline}
		\p
		The terms on the right in inequalities (\ref{conv1}) and (\ref{conv2}) do not depend on $X$ anymore so, for $N$ an integer there is $p_N$, independent from $X$ and with $p_N> p_{N-1}$, such that $1\leq\dfrac{\ell_X(\gamma_ {p_N})}{\ell_X(\gamma^{p_N}_{m_{p_N}})}\leq 1+\frac{1}{N}$  and $i(\hat{\gamma}^{p_N}_{m_{p_N}},\hat{\gamma}^{p_N}_{m_{p_N}})\leq \frac{1}{\pi^2|\chi|}(1+\frac{1}{N})$. As a consequence, we can take $\Curve{X}{N}=\gamma^{p_N}_{m_{p_N}}$. 
		
		\p The previous constructions ensure that $i(\hatCurve{X}{N},\hatCurve{X}{N})\leq \frac{1}{\pi^2|\chi|}(1+\frac{1}{N})$, and we have proved (1) in the statement of the theorem.
		
		\p Applying Proposition \ref{ipbound} we have $i_{per}(\Curve{X}{N},\Curve{X}{N})\leq 4e^{p_N}$ where $p_N$ does not depend on $X$, which gives us the third point.
		
		\p At last, (\ref{EgualIntersection}) and the choice of $p_N$ and $m_p$ induces that 
		\begin{equation*}
		1-\frac{3}{N}\leq (1-\frac{1}{N})\leq \frac{i(\hatCurve{X}{N},\alpha_q)}{\ell_X(\alpha_q)/\pi^2|\chi|}\leq(1+\frac{1}{N})^2\leq 1+\frac{3}{N},\quad \forall q\leq N,
		\end{equation*}
		hence, we obtain the second point with $n_\alpha=q$ when $\alpha=\alpha_q$.

		Moreover, up to passing to a subsequence, the $(\Curve{X}{N})_{N\in\BN}$ are built from the sequence $(\gamma_N)_{N\in\BN}$ of random geodesics through cutting processes of parameters $k_N=e^{p_N/2}\xrightarrow[N\to\infty]{}\infty$ and $\theta_N=e^{-p_N/2}\xrightarrow[N\to\infty]{}0$. As a consequence, Lemma~\ref{RandomGeod} ensures that we have built a sequence of random geodesics. At last, for $K$ a compact subsurface of $X$ we have
		\begin{equation*}
		i\left(\hatCurve{X}{n},\frac{\Curve{X}{n|K}}{\ell_X(\Curve{X}{n|K})}\right)\leq i(\hatCurve{X}{n},\hatCurve{X}{n}) \leq \dfrac{1}{\pi^2|\chi|}(1+\frac{1}{n}), 
		\end{equation*}
		and if we pass to the limit, using (\ref{compact}), we obtain that 
		$$\lim\limits_{N\to\infty}i(\hatCurve{X}{N},\hatCurve{X}{N})=\dfrac{1}{\pi^2|\chi|}.$$
	\end{proof}

	\section{Proof of Thurston's compactification}
	
	Armed with Theorem~\ref{Lemma}, we are now able to prove Thurston's compactification. As we already mentioned in the introduction, the starting point of this compactification is the embedding of $\CT(S)$ and $\BP_+\CM\CL(S)$ into the space $\BP_+(\BR_+^{\FC(S)})$:
	\begin{center} 
		$\begin{array}{ccccccccccc}
		\ell & : & \mathcal{T}(S) & \to     & \BP_+(\BR_+^{\FC(S)}) & \quad \\
		&   & X              & \mapsto & \BR_+\ell_X(\cdot),          \\
		
		\iota & : & \BP_+\CM\CL(S) & \to     & \BP_+(\BR_+^{\FC(S)}) \\
		&& \lambda      & \mapsto & \BR_+i(\lambda,\cdot).   \\
		\end{array}$
	\end{center}
	The image of $\CT(S)$ in $\BP_+(\BR_+^{\FC(S)})$ is included into a compact set (use \cref{MajorLength} for instance), thus, the closure $\overline{\CT}(S)$ of $\CT(S)$ is compact. The boundary of this set is given by the following theorem. 
	\begin{theo*}[Thurston's compactification]\label{Thurston}
		If $S$ is a finite analytic type surface with negative Euler characteristic then the accumulation points of $\CT(S)$ in $\BP_+(\BR_+^{\FC(S)})$ are the projective classes of functions $\gamma\mapsto i(\lambda,\gamma)$ where $\lambda\in\CM\CL(S)$ is a measured lamination on $S$.
	\end{theo*}
	
	Our arguments apply to the compact case, but for the sake of concreteness we will focus on non-compact surfaces.

	Let $X_k\in\CT(S)$  be a sequence  which converges in $\BP_+(\BR_+^{\FC(S)})$ and leaves all compact sets of $\CT(S)$, meaning that there are a non-zero element $F$ of $\BR_+^{\FC(S)}$ and a sequence $(\varepsilon_k)_{k\in\BN}$ of positive real numbers such that $\lim\limits_{k\to\infty} \varepsilon_k\ell_k(\cdot)=F$ pointwise (we have written $\ell_k$ for $\ell_{X_k}$). We will prove that $F$ is given by taking the intersection number with a suitable measured lamination. 
	
	Fix a filling curve $\beta$ on $S$, that is a closed curve such that the connected components of $S\setminus \beta$ are balls and annular neighborhoods of the cusps. Such a curve gives us a bound on the length of every curve $\gamma\in\FC(S)$, namely, 
	\begin{equation} \label{MajorLength}
	\ell_X(\gamma)\leq \ell_X(\beta) i(\gamma,\beta)(1+i(\gamma,\gamma))
	\end{equation}
	for every hyperbolic structure $X$ \cite[Lem. 2.1]{STH}. Since $F=\lim\limits_{k\to\infty} \varepsilon_k\ell_k(\cdot)$ is non-zero, there is $\gamma\in\FC(S)$ with $F(\gamma)\neq 0$. We obtain from (\ref{MajorLength}) that $0<F(\gamma)\leq F(\beta)(1+i(\gamma,\gamma))i(\gamma,\beta)$ and hence that $F(\beta)\neq 0$. Since we are only interested in convergence in  $\BP_+(\BR_+^{\FC(S)})$, we can assume that $F(\beta)=1$, meaning that
	\begin{align*}
	\lim\limits_{k\to\infty}\delta^k\frac{\ell_k(\cdot)}{\pi^2|\chi|}=F,
	\end{align*}
	where $\delta^k=\frac{\pi^2|\chi|}{\ell_k(\beta)}$.
	
	We will now prove that $F$ is of the form $i(\mu,\cdot)$ where $\mu$ is a measured lamination on $S$.

	Applying Theorem~\ref{Lemma} to each $X_k$, we obtain some sequences of essential closed geodesics $(\Curve{k}{n})_{n\in\BN}=(\hatCurve{X_k}{n})_{n\in\BN}$ with $\lim\limits_{n\to\infty}i(\Curve{k}{n}/\ell_k(\Curve{k}{n}),\cdot)=\ell_k(\cdot)/\pi^2|\chi|$. As all along, let $\Sigma$ be a compact complete hyperbolic surface with boundary whose interior is homeomorphic to $S$ and let's identify $\FC(S)$ with $\FC(\Sigma)$. In particular, we can consider the weighted curves $\hatCurve{k}{n}=\Curve{k}{n}/\ell_k(\Curve{k}{n})$ as currents of $\Sigma$. The space $\BP_+\CC(\Sigma)$ being compact each $(\hatCurve{k}{n})_{k\in\BN}$ projectively converges to a non-zero current $\mu_n\in\CC(\Sigma)$.

	We first want to show that the $\mu_n$ are measured laminations. Consider the sequence $(\hatCurve{k}{n})_{k\in\BN}$ for $n$ fixed, there are some $\varepsilon^k_n>0$ such that $\varepsilon^k_n\hatCurve{k}{n}$ tends to $\mu_n$ up to a subsequence in $k$. So, by diagonal extraction we can suppose that $\varepsilon^k_n\hatCurve{k}{n}\xrightarrow[k\to\infty]{}\mu_n$ for every $n$. What we have to show is that $\lim\limits_{k\to\infty}\varepsilon_n^k = 0$ for every $n$. 
	\p 
	The sequence $(X_k)_{k\in\BN}$ leaves every compact set of $\CT(S)$ so there is a simple closed curve $\alpha$ such that $\lim\limits_{k\to\infty} \ell_k(\alpha)=\infty$. Recall that to prove Theorem~\ref{Lemma} we have enumerated $\FC(S)=\{\alpha_n|n\in\BN\}$ such that $i_{per}(\alpha_n,\alpha_n)\leq 4n$, since $\alpha$ is a simple curve we can suppose that $\alpha=\alpha_1$. The $\Curve{k}{n}$ come from Theorem~\ref{Lemma} thus $\left| i(\hatCurve{k}{n},\alpha)\left(\frac{\ell_k(\alpha)}{\pi^2|\chi|}\right)^{-1}-1 \right|\leq\dfrac{3}{n}$ whatever $k$ and $n$. By hypothesis $\ell_k(\alpha)\xrightarrow[k\to\infty]{}\infty$ and we can suppose, up to a shift in $n$, that for every $n$, $\left| i(\hatCurve{k}{n},\alpha)\left(\frac{\ell_k(\alpha)}{\pi^2|\chi|}\right)^{-1}-1 \right|<\dfrac{1}{2}$. As a consequence $i(\hatCurve{k}{n},\alpha)\xrightarrow[k\to\infty]{}\infty$. However, $\infty>i(\mu_n,\alpha)
	=\lim\limits_{k\to\infty}\varepsilon^k_n i(\hatCurve{k}{n},\alpha)$ thus $\varepsilon^k_n\xrightarrow[k\to\infty]{}0$ for every $n$, and $i(\hatCurve{k}{n},\hatCurve{k}{n})$ is bounded independently from $k$ and $n$, hence, $i(\mu_n,\mu_n)=\lim\limits_{k\to\infty}(\varepsilon^k_n)^2i(\hatCurve{k}{n},\hatCurve{k}{n})=0$ and $\mu_n$ is a measured lamination on $\Sigma$. By construction, $i_{per}(\Curve{k}{n},\Curve{k}{n})\leq C_n$ for every $k$ and $n$, as mentioned earlier (or in \cite[Lem. 2.7]{ES}) it ensures that for $n$ fixed the $\Curve{k}{n}$ are all included in the same compact subsurface of $\Sigma\setminus\partial\Sigma$. It follows that $\mu_n$ is supported on a compact set of $\Sigma\setminus\partial\Sigma$ and by (\ref{CorresLam}) it is a measured lamination of $S$.
	
	Recall that $\beta$ is a filling curve of $S$, as a consequence, $i(\mu_n,\beta)\neq 0$ and hence, we can suppose that $i(\mu_n,\beta)=1$ for every $n$ and we obtain
	\begin{equation}\label{convergence}
	\lim\limits_{k\to\infty} \delta^k_n \hatCurve{k}{n} =\mu_n \text{ in } \CC(\Sigma), 
	\end{equation}
	where $\delta^k_n=\frac{1}{i(\beta,\hatCurve{k}{n})}$ is well-defined.

	To sum up, we have the following convergence diagram, where all the convergences are pointwise.
	\begin{center}
		\begin{tabular}{ccccccc}
			$\delta^1\dfrac{\ell_{X_1}(.)}{\pi^2|\chi|}$ & $\delta^2\dfrac{\ell_{X_2}(.)}{\pi^2|\chi|}$ & $\cdots$ & $\delta^k\dfrac{\ell_{X_k}(.)}{\pi^2|\chi|}$ & $\cdots$ & $\xrightarrow[\quad]{ } $ & $F\in\BR_+^{\FC(S)}$ \\
			$\big\uparrow$                & $\big\uparrow$                & $\cdots$ & $\big\uparrow$                & $\cdots$ &                           & $\big\uparrow$?      \\
			$\vdots$                      & $\vdots$                      & $\cdots$ & $\vdots$                      & $\cdots$ &                           & $\vdots$             \\
			$\delta^1_ni(\hatCurve{1}{n},\cdot)$ & $\delta^2_ni(\hatCurve{2}{n},\cdot)$  & $\cdots$ & $\delta^k_ni(\hatCurve{k}{n},\cdot)$  & $\cdots$ & $\xrightarrow{\quad}$     & $i(\mu_n,\cdot)$              \\
			$\vdots$                      & $\vdots$                      & $\cdots$ & $\vdots$                      & $\cdots$ & $\vdots$                  & $\vdots$             \\
			$\delta^1_2i(\hatCurve{1}{2},\cdot)$  & $\delta_2^2i(\hatCurve{2}{2},\cdot)$  & $\cdots$ & $\delta^k_2i(\hatCurve{k}{2},\cdot)$  & $\cdots$ & $\xrightarrow{\quad}$     & $i(\mu_2,\cdot)$              \\
			$\delta_1^1i(\hatCurve{1}{1},\cdot)$  & $\delta^2_1i(\hatCurve{2}{1},\cdot)$  & $\cdots$ & $\delta^k_1i(\Curve{k}{1},\cdot)$  & $\cdots$ & $\xrightarrow{\quad}$     & $i(\mu_1,\cdot)$              \\
		\end{tabular}
	\end{center}
	We want $F$ to be the pointwise limit of $(i(\mu_n,\cdot))_{n\in\BN}$. To prove it, it is sufficient to show that the convergence $\delta^k_ni(\hatCurve{k}{n},\gamma)\xrightarrow[n\to \infty]{ }\delta^k\frac{\ell_k(\gamma)}{\pi^2|\chi|}$ is uniform in $k$ when $\gamma\in\FC(S)$ is fixed.
	
	If $\gamma\in\FC(S)$ is fixed then Theorem~\ref{Lemma} ensures that $\left| \dfrac{\delta^k_n}{\delta^k} -1\right|\leq \epsilon_n$ and $ \left| i(\hatCurve{k}{n},\gamma)\left(\dfrac{\ell_k(\gamma)}{\pi^2|\chi|}\right)^{-1}-1 \right|\leq\epsilon_n$ for every $k$ and for $n$ large enough ($n_\gamma$ and $n_\beta$ do not depend on $k$) with $\epsilon_n \xrightarrow[n\to\infty]{}0$. Moreover, fixing $\gamma$ we know that $\delta^k\frac{\ell_k(\gamma)}{\pi^2|\chi|}\xrightarrow[k\to\infty]{}F(\gamma)$ hence the sequence $(\delta^k\frac{\ell_k(\gamma)}{\pi^2|\chi|})_{k\in\BN}$ is bounded by some $d_\gamma$ and we obtain
	\begin{equation*}
	\left| \delta^k_ni(\hatCurve{k}{n},\gamma)-\delta^k\frac{\ell_k(\gamma)}{\pi^2|\chi|}\right| \leq v_nd_\gamma\xrightarrow[n\to\infty]{}0.
	\end{equation*}
	Hence, the convergence holds uniformly in $k$, and $\lim\limits_{n\to\infty}\lim\limits_{k\to\infty} \delta^k_ni(\hatCurve{k}{n},\gamma) = \lim\limits_{k\to\infty}\lim\limits_{n\to\infty} \delta^k_ni(\hatCurve{k}{n},\gamma)$, which implies that $F(\gamma)=\lim\limits_{n\to\infty}i(\mu_n,\gamma)$. Moreover, $\CM\CL(S)$ is a closed subset of $\BR_+^{\FC(S)}$ hence $F(\cdot)=\lim\limits_{n\to\infty}i(\mu_n,\cdot)$ is of the form $F(\cdot)=i(\mu,\cdot)$ where $\mu\in\CM\CL(S)$, which was what we needed to prove. \qed
	
	\bibliographystyle{plain}
	\bibliography{Bibliography.bib}
	
\end{document}